\documentclass[11pt]{amsart}
\usepackage[utf8]{inputenc}
\usepackage[T1]{fontenc}
\usepackage{amsmath}
\usepackage{amsfonts}
\usepackage{amssymb}
\usepackage{amsthm}
\usepackage{geometry}
\usepackage{graphicx}
\graphicspath{{./images/}}
\usepackage{eso-pic}
\usepackage{url}
\usepackage[all]{xy}
\usepackage{esint}
\usepackage[pdftex,
            colorlinks=true,
            linkcolor=blue,
            citecolor=red,
            pdfauthor={Alexis de Villeroché},
            pdftitle={stabilitylambda2},
            pdfcreator={pdflatex}]{hyperref}
\usepackage{bm}
\usepackage{dsfont}
\usepackage{parskip}

\usepackage[backend = biber, sorting = nyt, citestyle = numeric,  bibstyle = numeric, maxbibnames=10]{biblatex}
\definecolor{vert}{RGB}{20,150,100}
\definecolor{rouge}{RGB}{250,50,50}

\newcommand{\R}{\mathbb{R}}
\newcommand{\Rd}{\R^d}
\newcommand{\Co}{\mathcal{C}^0}

\newcommand{\btq}{~ \big| ~}
\renewcommand{\phi}{\varphi}
\newcommand{\eps}{\varepsilon}
\newcommand{\A}{\mathcal{A}}
\renewcommand{\O}{\Omega}
\newcommand{\be}{\begin{equation}}
\newcommand{\ee}{\end{equation}}

\numberwithin{equation}{section}

\theoremstyle{plain}

\newtheorem{trm}{Theorem}[section]
\newtheorem{lem}[trm]{Lemma}
\newtheorem{cor}[trm]{Corollary}
\newtheorem{prop}[trm]{Proposition}

\theoremstyle{remark}
\newtheorem{rem}[trm]{\bf Remark}

\title[Full spectrum control by the second eigenvalue]{Quantitative stability control of the full spectrum of the Dirichlet Laplacian by the second eigenvalue}
\author{Alexis de Villeroché}
\date{}

\begin{document}
\begin{abstract}
Let $\Omega\subset \mathbb{R}^d$ be an open set of finite measure and let $\Theta$ be a disjoint union of two balls of half measure. We study the stability of the full Dirichlet spectrum of $\Omega$ when its second eigenvalue is close to the second eigenvalue of $\Theta$. Precisely, for every integer $k \ge 1$,  we provide a quantitative control of the difference $|\lambda_k(\Omega)-\lambda_k(\Theta)|$ by the variation of the second eigenvalue $C(d,k)(\lambda_2(\Omega)-\lambda_2(\Theta))^\alpha$, for a suitable exponent $\alpha$ and a positive constant $C(d,k)$ depending only on the dimension of the space and the index $k$. We are able to find such an estimate for general $k$ and arbitrary $\Omega$  with $\alpha =\alpha_d/(d+1)^2$ where $\alpha_2 = 1/2$ and $0<\alpha_d<1$ in higher dimensions. In the particular case where $\lambda_k(\Omega)\ge \lambda_k(\Theta)$,  we can improve the inequality and find an estimate with the sharp exponent $\alpha = 1/2$.
\end{abstract}
\maketitle

\section{Introduction}\label{Intro}
We work in the Euclidean space $\Rd$, for some $d\ge 2$. Let us denote by $\omega_d$ the volume of the unit ball of $\Rd$ and  set
$$
\A = \big\{\Omega \subset \Rd ~\big|~\Omega\text{ open and } |\Omega|=\omega_d\big\}.
$$
In the following, by $B\in \A$ we  denote a  ball of radius equal to $1$ and by $\Theta \in \A$ a union of two disjoint balls of measure $\omega_d/2$. The position of the balls does not affect the spectrum of the Dirichlet Laplacian, however, in our analysis their position may play a role. This will be specified, when necessary.

For any $\Omega\in \A$, let us consider the k-th eigenvalue of the Dirichlet Laplacian, multiplicity being counted. For $k\ge 1$,
\begin{equation}
\lambda_k(\Omega) =\min_{S_k\subset H^1_0(\Omega)} \sup_{u\in S_k,\,u\neq 0} \frac{\int_\Omega |\nabla u|^2}{\int_\Omega u^2},
\end{equation}
where $S_k$ is a subspace of $H^1_0(\Omega)$ of dimension $k$. The associated eigenfunctions which achieve the minimum are denoted by $u_{\O,k}$ and solve 
\begin{equation}
\left\{\begin{array}{ll}
-\Delta u_{\O,k} = \lambda_k(\Omega)\, u_{\O,k} &\text{ in }\Omega,\\
u_{\O,k} = 0 &\text{ in } \partial \Omega.
\end{array}\right.
\end{equation} 
If not otherwise specified,  we consider them  normalized in $L^2(\Omega)$.

Minimizing $\lambda_k(\Omega)$ for $\Omega \in \A$ is an important question in spectral geometry, as it is related to the celebrated P\' olya conjecture stating 
$$\forall k \ge 1, \forall \; \Omega \in \A, \;\; \lambda_k(\Omega) \ge \frac{4\pi^2 k^{2/d}}{\omega_d^{4/d}}.$$

One strategy to understand the conjecture is, for given integer $l \ge 1 $, to solve the minimization problem
$$\min\{\lambda_l(\Omega) : \Omega \in \A\},$$
precisely to characterize solutions and find their properties. The key element, is that the P\'olya conjecture is known to hold for some {\it particular} class of domains and so, a suitable characterization of the solution may provide useful information.

Assume $\Omega_l^*$ is a minimizer of $\lambda_l$ in $\A$ and $\Omega \in \A$ is a set such that $\lambda_l(\Omega)$ approaches the minimal value $\lambda_l(\Omega_l^*)$.
We aim for a sharp control of the variation of the $k$-th eigenvalue by the variation of the $l$-th eigenvalue. Precisely,  is the following inequality true
\be\label{jo01}
\forall k\ge 1, \forall \Omega \in \A, \; \;  |\lambda_k(\Omega)-\lambda_k(\Omega_l^*)|\le C(d,k) [\lambda_l(\Omega)-\lambda_l(\Omega_l^*)]^\alpha \; ?
\ee
Above, $\alpha$ is a suitable dimensional exponent and $C(d,k)$ is a constant depending only on the dimension and $k$. Implicitly, such an inequality implies either uniqueness of $\O_l$ or isospectrality of all minimizers.

At the moment, the minimizers   $\Omega_l^*$ are analytically known only for $l=1$ and $l=2$ and are respectively the ball $B$ and the  union of two disjoint equal balls, that we generally denote $\Theta$. This is a consequence of  the Faber-Krahn  and the Krahn-Szeg\"o inequalities (see for example \cite{H06} and \cite{HP05}).

For $l\ge 3$ some qualitative results are known. For instance,  Bucur \cite{B12} and Mazzoleni and Pratelli \cite{MP13} proved that a minimizer exists in the larger class of quasi-open sets of measure $\omega_d$. The minimizers are bounded and of finite perimeter. Moreover, the structure of their reduced boundary was analyzed by Kriventsov and Lin in \cite{KrLi18, KrLi19}, but the qualitative information about their global geometry is missing. Thus, since so little is known in the case $l\ge 3$, we expect to be able to find stability estimates like those in \eqref{jo01} only for the cases $l=1$ and $l=2$.

The case $l=1$ has already been studied in the literature. As we already stated, we know that the ball $B$ is the minimizer to $\lambda_1$ as asserted by the Faber-Krahn inequality
\begin{equation}
\forall \Omega \in \A,\qquad\lambda_1(\Omega)\ge \lambda_1(B),
\end{equation}
with equality if and only if $\Omega$ is a ball.

In 2006, Bertrand and Colbois in \cite{BC06} where able to prove stability estimates near the ball. Precisely, they prove that for any $\Omega \in \A$ such that
$$\lambda_1(\Omega)\le (1+\eps)\lambda_1(B),$$
it holds, if $\eps< \eps_k$ small enough,
\begin{equation}
|\lambda_k(\Omega)-\lambda_k(B)| \le C_{d,k}\  \eps^\frac{1}{80\,d},
\end{equation}
Clearly, the exponent $\frac{1}{80\,d}$ is not optimal, and later in 2019, Mazzoleni and Pratelli \cite{MP19} improved the result and obtained for $\lambda_1(\Omega) \le \lambda_1(B) + 1$ and for any $\eta>0$
\begin{equation}
- C_{d,k, \eta} (\lambda_1(\Omega)-\lambda_1(B))^{\frac{1}{6}-\eta}\le \lambda_k(\Omega)-\lambda_k(B)\le C_{d,k,\eta} (\lambda_1(\Omega)-\lambda_1(B))^{\frac{1}{12}-\eta},
\end{equation}
and they are able to improve the exponent in dimension $2$.\\
The sharp estimate has been obtained in 2023, by Bucur, Lamboley, Nahon and Prunier in \cite{BLNP23}, where they prove for any $\Omega\in\A$ and  $k$
\begin{equation}
\label{eq:lbd1}
|\lambda_k(\Omega)-\lambda_k(B)|\le C_d k^{2+\frac{4}{d}} \lambda_1(\Omega)^\frac{1}{2}\,(\lambda_1(\Omega)-\lambda_1(B))^\frac{1}{2}
\end{equation}
A key ingredient of the proof is the  quantitative Faber-Krahn inequality proved in 2015 by Brasco, de Philippis and Velichkov in \cite{BPV15}, 
\begin{equation}
\label{eq:qFKI}
\forall \Omega \in \A, \qquad\lambda_1(\Omega)\ge \lambda_1(B)\,(1+C_d\,\mathcal{F}_1(\Omega)^2).
\end{equation}
where $\mathcal{F}_1(\Omega)$ denotes the \textit{Fraenkel asymmetry} 
$$
\mathcal{F}_1(\Omega) = \min\left\{\frac{|\Omega\Delta B|}{|\Omega|} ~\middle|~ B\subset \Rd \text{ ball with } |B| = |\Omega|\right\}.
$$
We point out that inequality \eqref{eq:lbd1} can be improved for some specific values of $k$. Precisely, if $\lambda_k(B)$ is simple, then the exponent $1/2$ on the right hand side can be replaced by the exponent $1$. This result is very fine and relies on the analysis of a degenerate free boundary problem of vectorial type. The key point is that if $\lambda_k(B)$ is simple then the ball is a critical set for $\lambda_k$, as for $\lambda_1$. Intuitively this makes that the variation of $\lambda_k$ is of the same order as the variation of $\lambda_1$. \\
Another related work by Allen, Kriventsov and Neumayer (see \cite{AKN25}) concerns the stability in the Faber Krahn inequality but with respect to some norm $\|.\|$ of the resolvent operator $R_\O$ instead of the spectrum. They show for a well positioned ball $B$,
\be 
\forall \O \in \A,\qquad \lambda_1(\O) - \lambda_1(B) \ge C_d \|R_\O-R_B\|^2 
\ee
and from this first result, they deduce that the deficit in the first eigenvalue quantitatively controls the deficit in $L^2$ norm of the $k$-th eigenfunctions, 
\be 
\|u_{\O,k} - u_{B,k}\|_{L^2} \le C_{k,d} (\lambda_1(\O)-\lambda_1(B))^\frac12
\ee
where the exponent $1/2$ is sharp.

Our goal is to analyze the case $l=2$ and to prove an inequality similar to \eqref{eq:lbd1} with the variation of the second eigenvalue on the right hand side, variation with respect to its global minimizer which is  the disjoint union of two balls of half measure, $\Theta$. 
Indeed, we know from the Krahn-Szeg\"o inequality that $\Theta$ minimizes $\lambda_2$ in $\A$,
\begin{equation}
\label{eq:KSI}
\forall \Omega\in\A,\qquad\lambda_2(\Omega)\ge \lambda_2(\Theta),
\end{equation}
with equality if and only if $\Omega =\Theta$.
Moreover, a quantitative inequality has been proved by Brasco and Pratelli \cite[Theorem 3.5]{BP12} in 2013.
\begin{equation}
\label{eq:qKSI}
\forall \Omega \in \A,\qquad \lambda_2(\Omega)\ge \lambda_2(\Theta)\,(1+C_d\,\mathcal{F}_2(\Omega)^{d+1}),
\end{equation}
Here, $\mathcal{F}_2(\Omega)$ is the \textit{Fraenkel 2-asymmetry} defined as follows for any open set $\Omega \subset \Rd$
\be \label{2Fasym}
\mathcal{F}_2(\Omega) = \inf\left\{\left.\frac{|\Omega\Delta(B_1\cup B_2)|}{|\Omega|}\right|\,B_{1}, B_{2} \subset \Rd \text{ balls, } \,|B_1\cap B_2| = 0 \text{ and } |B_{1}|=|B_{2}|=\frac{|\Omega|}{2}\right\}.
\ee
The exponent given for Inequality \eqref{eq:qKSI} is $d+1$. Indeed, following the remark of the authors about the exponent they obtained in their proof \cite[Remark 3.6]{BP12}, at the time of release of their paper the sharp inequality \eqref{eq:qFKI} was not proved to hold  and it improves the exponent from $2(d+1)$ to $d+1$. They also point out that whether the exponent $d+1$ is sharp is not clear but that the sharp exponent has to be larger than $\frac{d+1}{2}$ (see also the discussion in \cite[Remark 7.54-7.56]{BdP17}).

Below we state our main results.
\begin{trm}
\label{TH1}
There exists two dimensional constants $C_d, \alpha_d>0$ such that for any $\Omega \in \A$, and for any $k\ge 1$, it holds,
$$
\left|\lambda_k(\Omega)-\lambda_k(\Theta)\right|\le C_d\,k^{2+\frac{4}{d}}\,\lambda_2(\Omega)^{1- \frac{\alpha_d}{(d+1)^2}}\,(\lambda_2(\Omega)-\lambda_2(\Theta))^\frac{\alpha_d}{(d+1)^2}.
$$
Moreover $\alpha_d < 1$.
\end{trm}
In dimension $2$ we justify that $\alpha_2 = 1/2$ is convenient, however not sharp (see Section \ref{sec2}, Remark \ref{rm:holderdim2}). In higher dimensions the values of $\alpha_d$ are not identified, they are related to Hölder estimates for the torsion function of a specific class of open sets (described in Section \ref{sec2}, Proposition \ref{prop:holderestimate}) which satisfy a uniform capacity density condition.\\
In general, we do not expect the exponent $\frac{\alpha_d}{(d+1)^2}$ to be sharp. We are able to improve the exponent to the sharp one $1/2$ in the case that $\lambda_k(\Omega) \ge \lambda_k(\Theta)$. Our second result reads
\begin{trm}
\label{TH2bis}
There exists a dimensional constant $C_d>0$ such that for any $\Omega \in \A$, and for any $k\ge 1$, it holds,
$$
\left(\lambda_k(\Omega)-\lambda_k(\Theta)\right)_+\le C_d\,k^{2+\frac{4}{d}}\,\lambda_2(\Omega)^{\frac{1}{2}}\,(\lambda_2(\Omega)-\lambda_2(\Theta))^\frac{1}{2},$$
where by $(a)_+$ we denote the positive part of the number $a$.
\end{trm}
As an intermediate technical step, we prove the following more general result.
\begin{trm}
\label{TH2}
There exists a dimensional constant $C_d>0$ such that for any $\Omega \in \A$ and any pair of disjoint open subsets $\Omega^+$ and $\Omega^-$ of  of $\Omega$, verifying $\lambda_2(\Omega) \ge \max(\lambda_1(\Omega^+),\lambda_1(\Omega^-))$, it holds for all $k\ge 1$
$$
|\lambda_k(\Omega^+\cup\Omega^-) - \lambda_k(\Theta)| \le  C_d\,k^{2+\frac{4}{d}}\,\lambda_2(\Omega)^{\frac{1}{2}}(\lambda_2(\Omega)-\lambda_2(\Theta))^\frac{1}{2}.
$$
\end{trm}
Notice that such couples of subsets $\Omega^+$, $\Omega^-$ indeed do exist. For instance, we can choose them to be the two nodal sets of the second eigenfunction of $\Omega$, but other choices could  be more relevant in specific situations.\\
Another direct corollary of Theorem \ref{TH2} is that the full sharp inequality holds true for disconnected sets in $\A$ when the two disjoint components are sufficiently equilibrated.
\begin{cor}
Let $\O\in\A$, such that $\O = \O_1\sqcup\O_2$ is disconnected satisfying $\lambda_1(\O)= \lambda_1(\O_1)$ and $\lambda_2(\O) = \lambda_(\O_2)$, then it holds for any $k\ge 1$, 
$$
|\lambda_k(\O) - \lambda_k(\Theta)| \le  C_d\,k^{2+\frac{4}{d}}\,\lambda_2(\Omega)^{\frac{1}{2}}(\lambda_2(\Omega)-\lambda_2(\Theta))^\frac{1}{2}.
$$
\end{cor}

\section{Preliminary results}\label{sec2}
To fix the terminology, below we recall some results, useful along the proofs. 

We begin with the definition and a few properties of the torsional rigidity and of the torsion function. For any $\Omega\in \A$ let us define the torsional rigidity of $\Omega$, $T(\Omega)$ as
\begin{equation}
T(\Omega) = \max_{u\in H^1_0(\Omega)} \int_\Omega 2\,u -\int_\Omega |\nabla u|^2.
\end{equation}
The function which achieves the maximum is the unique weak solution in $H^1_0(\Omega)$ of  
\begin{equation}
\left\{\begin{array}{ll}
-\Delta u= 1 &\text{ in }\Omega,\\
u = 0 &\text{ in } \partial \Omega.
\end{array}\right.
\end{equation} 
We denote the solution by $w_\Omega$ and call it torsion function. 
Then
$$
T(\Omega) =  \int_\Omega 2\,w_\Omega -\int_\Omega |\nabla w_\Omega|^2 =\int_\Omega w_\Omega.
$$
We recall the Saint-Venant inequality (see for instance \cite{H06} and \cite{HP05})
\begin{equation}
\label{eq:SVI}
\forall \Omega \in \A,\qquad T(\Omega)\le T(B),
\end{equation}
with equality if and only if $\Omega$ is a ball, up to a set of zero capacity.

And we also recall the Talenti inequality (see \cite[Theorem 2]{T76}) that states that the supremum of the torsion function is also maximized for the ball 
\begin{equation}
\label{eq:Ti}
\forall \Omega \in \A,\qquad \|w_\Omega\|_{L^\infty(\Omega)}\le \|w_B\|_{L^\infty(B)}.
\end{equation}

We give now a lemma in the same spirit as \cite[Lemma 2.2]{BM15} that states that the torsion function of an open set $\O$ can be locally controlled by its mean value.
\begin{lem} \label{lm:subharmonic}
For every open subset $\O\subset \Rd$
$$ \forall r>0,\ x_0\in \O \,\ w_\O(x_0) \le \frac{1}{|B(x_0,r)|} \int_{B(x_0,r)} w_\O(x) \mathrm dx + \frac{r^2}{2d} $$
\end{lem}
\begin{proof}
Let $\O$ be an open subset of $\Rd$ and $w_\O$ its torsion function extended by $0$ outside of $\O$, its Laplacian satisfies in the distributional sense in $\Rd$, $-\Delta w_\O \le 1$ therefore for every $x_0 \in \O$ the function $w_\O + \frac{|x-x_0|^2}{2d}$ is subharmonic in $\Rd$, so that 
$$ \forall r>0, \ w_\O(x_0) \le \frac{1}{|B(x_0,r)|} \int_{B(x_0,r)} \left(w_\O + \frac{|x-x_0|^2}{2d}\right) \le \frac{1}{|B(x_0,r)|} \int_{B(x_0,r)} w_\O(x)\,\mathrm dx + \frac{r^2}{2d} $$
\end{proof}

We recall from \cite[Theorem 3.4]{B03} (and \cite[Lemma 2.3]{BLNP23} for the precise dependence in $k$) the following estimate, which  shows that one can control the difference of the eigenvalues by the difference in torsional rigidities, between an open set $\Omega$ and a subset $\Omega'$.
\begin{lem}
\label{lm:B}
For any $\Omega\in\A$ and $\Omega'\subset\Omega$ open, it holds for all $k\ge 1$,
$$
0\le \frac{1}{\lambda_k(\Omega)}-\frac{1}{\lambda_k(\Omega')}\le \exp(1/(4\pi))\,k\,\lambda_k(\Omega)^{d/2} (T(\Omega)-T(\Omega')).
$$
\end{lem}
We also quote a result by Cheng and Yang (see \cite[Theorem 3.1]{CY07}) which comes as an improvement of an older less general result form Ashbaugh and Benguria \cite{AB94} that gives a control of the maximal ratio between the $k$-th and the first eigenvalues of a given open set.
\begin{lem}
\label{lm:CY}
For any $\Omega \in\A$, and for all $k\ge 1$,
$$
\lambda_k(\Omega)\le \left(1+\frac{4}{d}\right) \,k^\frac{2}{d}\,\lambda_1(\Omega).
$$
\end{lem}
We also recall the Kohler-Jobin inequality for any open set of finite measure $\Omega\subset \R^d$, see \cite{KJ78I} and \cite{KJ78II},
\begin{equation}
\label{KJ1}
\lambda_1(\Omega)^\frac{d+2}{2} T(\Omega) \ge \lambda_1({B})^\frac{d+2}{2}T({B}).
\end{equation}
Notice here that the exponent $\frac{d+2}{2}$ is such that the functional $\Omega \mapsto \lambda_1(\Omega)^\frac{d+2}{2} T(\Omega)$ is scale invariant.

We present below some results that are specific to the second eigenvalue $\lambda_2$. They play a fundamental role in the proof of Theorem  \ref{TH1}. For $l \ge 3$, such results are not available, so that it is not possible to adapt the proof of Theorem \ref{TH1} to higher eigenvalues $\lambda_l$ with $l\ge 3$.

We start by recalling a decomposition result for the second eigenvalue proved in \cite[Lemma 3.1]{BP12}.
\begin{lem}
\label{lm:decomp}
Let $\Omega \in \A$. There exists two disjoint subsets $\Omega^+,\Omega^-\subset \Omega$ such that 
$$\lambda_2(\Omega)= \max\{\lambda_1(\Omega^+),\lambda_1(\Omega^-)\}.
$$
\end{lem}
Note that if $\Omega$ is connected, then $\Omega^+$ and $\Omega^-$ correspond to the nodal sets of an eigenfunction $u_2$ associated to $\lambda_2(\Omega)$, namely $\Omega^+=\{u_2>0\}$, $\Omega^-=\{u_2<0\}$. In that case $\lambda_2(\Omega)= \lambda_1(\Omega^+)=\lambda_1(\Omega^-)$.  If $\Omega$ is disconnected, the construction of $\Omega^+,\Omega^-$ may (or not) use two different connected components, possibly making that the difference $\Omega\setminus (\Omega^+\cup \Omega^-)$ is a set of positive measure. 

We now recall a Kohler-Jobin type inequality for the second eigenvalue
\begin{equation}
\label{lm:KJ2}
\forall \Omega\subset \A,\qquad \lambda_2(\Omega)^\frac{d+2}{2} T(\Omega) \ge \lambda_2(\Theta)^\frac{d+2}{2}T(\Theta).
\end{equation}
This inequality comes as a consequence of \cite[Lemma 6]{vdBI13}, yet we give it here a simple proof for the sake of completeness.

\begin{proof}
Take $\Omega$ in $\A$, by Lemma \ref{lm:decomp} there exists two disjoint open subsets of $\Omega$, $\Omega^+$ and $\Omega^-$ such that $\lambda_2(\Omega) = \max \{\lambda_1(\Omega^+),\lambda_1(\Omega^-)\}$ and also, by inclusion, $T(\Omega) \ge T(\Omega^+)+T(\Omega^-)$. Then, take  $B^{1/2}$ a ball of volume $\omega_d/2$. We have by the Kohler-Jobin inequality (\ref{KJ1}),
$$
\begin{aligned}
\lambda_1(\Omega^+)^\frac{d+2}{2} T(\Omega^+) &\ge \lambda_1(B^{1/2})^\frac{d+2}{2}T(B^{1/2}),\\
\lambda_1(\Omega^-)^\frac{d+2}{2} T(\Omega^-) &\ge \lambda_1(B^{1/2})^\frac{d+2}{2}T(B^{1/2}).
\end{aligned}
$$
Then, we can compute 
$$
\begin{aligned}
 \lambda_2(\Omega)^\frac{d+2}{2} T(\Omega) &\ge \max \{\lambda_1(\Omega^+),\lambda_1(\Omega^-)\}^\frac{d+2}{2}\,(T(\Omega^+)+T(\Omega^-))\\
 &\ge \lambda_1(\Omega^+)^\frac{d+2}{2} T(\Omega^+) + \lambda_1(\Omega^-)^\frac{d+2}{2} T(\Omega^-)\\
 &\ge 2\,\lambda_1(B^{1/2})^\frac{d+2}{2}T(B^{1/2})\\
 &=\lambda_2(\Theta)^\frac{d+2}{2}T(\Theta).
\end{aligned}
$$
\end{proof}
\begin{rem}  Note that a Kohler-Jobin type inequality for higher order eigenvalues would have its own interest. However, it is very unclear even for $l =3$  if a minimizer does exist  for  $\lambda_3(\Omega)^\frac{d+2}{2}\,T(\Omega)$ in the class $\A$.
\end{rem} 
\begin{rem} 
Yet the case of the associated maximization problem for the first eigenvalue 
$$
\max\big\{\lambda_1(\Omega)^p \,T(\Omega)~\big| ~ \Omega \in \A\big\},
$$
has been studied in \cite{BBGlB23} by Briani, Buttazzo and Guarino Lo Bianco, where they proved existence of an optimal shape for $p>p_1>1$ large enough. In \cite{BLNP23}, it was proved that the ball is the maximizer for $p>p_2$ for some $p_2$ larger than $p_1$. 
\end{rem}

Finally, let us finish this section by proving a Hölder continuity result for the torsion function of a specific class of open sets $\O\subset\Rd$.
\begin{prop}\label{prop:holderestimate}
Let $\O$ be an open subset of $\Rd$, of diameter $\mathrm{diam}(\O)=D$. Assume $\O$ has cylindrical symmetry of axis $\R \mathbf{e}$ for $\mathbf e$ a unit vector in $\Rd$ and is such that for any $z \in \R$ the intersection of $\Omega$ with the hyperplane $\mathbf e^\perp + z\mathbf e$ is either a ball in $\R^{d-1}$ of center $z\mathbf e$ or empty.\\
Then there exists constants $M_d(D) >0$ and $0<\alpha_d<1$ such that 
$$ \forall x\in \O,\ w_\O(x) \le M_d(D)\, \mathrm{dist}(x,\partial\O)^{\alpha_d} $$
\end{prop}

This result relies on the {\it Wiener criterion} for the complementary of $\O$, $\O^c = \Rd\setminus\O$. Let $A$ be an open subset of $\Rd$, and $K$ a compact subset of $A$, we define the capacity of $K$ relative to $A$ as 
\be\label{eq:Capacity}
\mathrm{Cap}(K,A) = \inf\Big\{ \int_A |\nabla u|^2\ \Big |\ u\in H^1_0(A),\ u\ge 1\text{ a.e. in a neighborhood of }K\Big\},
\ee
and for an open subset $A' \subset A$, we define its capacity as 
$$ \mathrm{Cap}(A',A) = \sup\Big\{ \mathrm{Cap}(K,A)\ \Big | \ K\subset A\text{, compact, }K\supset A'\Big\}.$$
For any set $E$, we define the capacity density function of a set $E$ in a point $z$ and $t>0$ as 
\be\label{def:phi}
\phi(z,E,t) = \frac{\mathrm{Cap}(E\cap B(z,t), B(z,2t))}{\mathrm{Cap}(B(z,t),B(z,2t))}.
\ee
We refer the interested reader to \cite[Chapter 2]{HKM06} and \cite[Chapter 2;7-9]{M85} for a thorough analysis of capacities and their properties.\\
the proof of Proposition \ref{prop:holderestimate} relies on the following theorem which is a weaker version of \cite[Theorem 6.18]{HKM06}, adapted to our setting.

\begin{trm}\label{th:osc}
Suppose that $\O$ is open bounded. Let $\theta \in H^1(\O)\cap\Co(\overline\O)$ and let $h$ be the unique harmonic function in $\O$ with $\theta - h \in H^1_0(\O)$.\\
If $x_0 \in \partial \O$, then for $0\le r\le \rho$, it is true that
$$
\mathrm{osc}(h,\O(r)) \le \mathrm{osc}(\theta ,\partial \O \cap \overline B(x_0,2\rho)) +  \mathrm{osc}(\theta, \partial \O)\, \exp(-c_d\int_r^\rho \phi(x_0, \O^c, t)\frac{\mathrm dt}{t}),
$$
where $\O(r) = \O\cap B(x_0,r)$, $c_d>0$.
\end{trm}

Here $\mathrm{osc}(f,A)$ denotes the {\it oscillations} of a function $f$ on a set $A$, defined as 
$$ \mathrm{osc}(f,A) = \sup_A f -\inf_A f. $$
We may now tackle the proof of Proposition \ref{prop:holderestimate}
\begin{proof}
Let $\O$ satisfy the hypotheses of Proposition \ref{prop:holderestimate} and fix $r>0$. Let $x_0\in\partial \O$ be a point on the boundary of $\O$, we write it as $x_0 = z\mathbf e + x$ with $x\in \mathbf e^\perp$, since the intersection of $\O$ with the hyperplane $z\mathbf e + \mathbf e^\perp$ is a ball in $\R^{d-1}$, the half space $H(x_0)^+ = \big\{z\mathbf e + y \btq y\in \mathbf e^\perp,\  y\cdot x \ge |x|\big\}$ is included in $\O^c$ so that for all $t\ge r$,
$$ 
\phi(x_0,\O^c,t) \ge \frac{\mathrm{Cap}(H^+(x_0)\cap B(x_0,t), B(x_0,2t))}{\mathrm{Cap}(B(x_0,t),B(x_0,2t))} = \frac{\mathrm{Cap}(H^+(x_0)\cap B(x_0,1), B(x_0,2))}{\mathrm{Cap}(B(x_0,1),B(x_0,2))} =\bar c_d,
$$
where parameter $\bar c_d$ is a constant independent of $x_0$.\\
Up to applying a translation to $\O$, we can assume that $0\in \O$, and we define $u = \frac{D^2-|x|^2}{2d}$ the torsion function of the ball $B(0,D)\supset \O$. We now apply Theorem \ref{th:osc} on $\O$ in $x_0$, with $\theta = u$, $h = u-w_\O$, and $\rho\ge r$, it writes using the lower bound on $\phi$, 
$$
\begin{aligned}
\mathrm{osc}(u-w_\O,\O(r)) \le & \mathrm{osc}(u ,\partial \O \cap \overline B(x_0,2\rho)) +  \mathrm{osc}(u, \partial \O)\, \exp(-c_d\int_r^\rho \phi(x_0, \O^c, t)\frac{\mathrm dt}{t})\\
\le& \mathrm{osc}(u ,\partial \O \cap \overline B(x_0,2\rho)) +\mathrm{osc}(u, \partial \O) \left( \frac r \rho\right)^{c_d \bar c_d}
\end{aligned}
$$
Then, since $u\ge w_\O$ on $\O(r)$ and from a direct computation with the expression of $u$, we obtain
$$
\begin{aligned}
\mathrm{osc}(u-w_\O,\O(r)) &\ge \sup_{\O(r)} w_\O,\\
\mathrm{osc}(u ,\partial \O \cap \overline B(x_0,2\rho)) &\le \frac{2D}{d}\,\rho,\\
\mathrm{osc}(u, \partial \O) &\le \frac{D^2}{2d}.
\end{aligned}
$$
Finally, we optimize in $\rho$ and obtain that there exists a constant $M_d(D)$ such that the property holds with $\alpha_d = \frac{c_d \bar c_d}{1+c_d\bar c_d}$.
\end{proof}
\begin{rem}\label{rm:holderdim2}
In dimension $2$, it is possible to show a similar and finer inequality with $\alpha_2=1/2$ and $M$ depending on the measure of $\O$ instead of the diameter for all simply connected $\O$. This relies on an $L^2$ estimate of the Green function and was already used in \cite[proof of Theorem 7.1, step 3, inequality (7.4)]{BB24}, we recall here in a few lines how to prove it. The starting point is the estimate given in \cite[Proof of Theorem 1.5, inequality (5.16)]{vdBB99}, if $\pi R_0^2 = |\O|$,
$$
\int_\O G_\O(x,y)^2 dy \le  \frac{8R_0}\pi \mathrm{dist}(x,\partial \O),
$$
and combine it with the Cauchy-Schwarz inequality 
$$
\forall x\in\O,\ w(x) \le |\O|^{1/2} \left(\int_\O G_\O(x,y)^2 dy\right)^{1/2} \le \frac{|\O|^{3/4} 8^{1/2}}{\pi^{3/4}} \mathrm{dist}(x,\partial \O)^{1/2}.
$$
In higher dimensions ($d\ge3$), it is possible to estimate the parameter $\bar c_d$. We know from \cite[Section 2.2.4, eq (1) p.106; Section 2.3.3 Lemma p.115]{M85},
$$
\begin{aligned}
\mathrm{Cap}(B(0,1),B(0,2)) = \frac{d\omega_d\,(d-2)}{1-{\frac12}^{d-2}}\\
\mathrm{Cap}(B(0,1)\cap H^+(0), B(0,2)) \ge \frac{d\omega_d}{2 \gamma_d},
\end{aligned}
$$
with
$$\gamma_d = \int_0^\infty \frac{1}{\cosh(x)^{d-2}}\mathrm dx \le \frac{2^{d-2}}{d-2}.$$
We deduce 
$$ \bar c_d \ge \frac{1}{2} \frac{1}{2^{d-2}}\left(1-\frac{1}{2^{d-2}}\right). $$
\end{rem}
{\bf Notation - }{\it In the following, $C_d$ will always denote a purely dimensional constant which may vary from line to line.}
\section{Proof of Theorem \ref{TH1}}\label{sec3}
The proof of Theorem \ref{TH1} is based on an approximation of the torsion of a set by the torsion of the best overlapping of the set by two disjoint equal balls of half measure. If such an approximation is rough in general, one can reasonably expect that the approximation becomes sharp if the set has a second eigenvalue close to its minimum, on the set $\Theta$. The set $\Theta$ is a generic union of two disjoint balls of half of the measure of $\Omega$, and their positioning has to be done in a optimal manner, we will actually choose them such that they are close to a minimum for the {\it Fraenkel 2-asymmetry}.

In \cite[Lemma 3.2]{BLNP23}, the authors got a control of the difference of the torsional rigidities between a set and the intersection of the set with a ball.
The following  lemma keeps the same spirit, and gives control of the difference  of the torsional rigidities between $\Omega$ and $\Omega\cap \Theta$ by the volume of the symmetric difference. Yet, its proof is much more delicate and leads to a weaker control as it relies on Proposition \ref{prop:holderestimate}. The main reason is the fact that two touching balls have a complement which is not locally a Lipchitz set. 
\begin{lem} \label{lm:torvol}
There exists a dimensional constant $C_d$ and such that for any $\Omega \in \A$, it holds
$$
T(\Omega)-T(\Omega\cap\Theta) \le C_d \ |\Omega\setminus \Theta|^\frac{\alpha_d}{d+1},
$$
where $\alpha_d$ is given by Proposition \ref{prop:holderestimate}.
\end{lem}
\begin{proof}
Let $\O$ be in $\A$ and fix $\Theta = B_1 \sqcup B_2$, with $B_1$ and $B_2$ two disjoint balls of measure $\omega_d/2$, and we denote by $w_\O$ the torsion function of $\O$.\\
First, if $\O$ does not intersect $\Theta$, the inequality is obviously true. If $|\O \cap B_2| = 0$, from \cite[Lemma 3.2]{BLNP23}, since the measure of the set $\O$ does not play a role in the proof, we obtain,
$$
T(\O) - T(\O\cap \Theta) \le C_d \ |\O\setminus B_1| = C_d \ |\O\setminus \Theta|.
$$
If $\O = \O_1 \sqcup \O_2$ is disconnected. Suppose first that $|\O_1\cap B_2|= |\O_2\cap B_1| = 0$, from \cite[Lemma 3.2 ]{BLNP23}, we obtain 
$$
\begin{aligned}
T(\O) - T(\O\cap \Theta) &= T(\O_1) - T(\O_1\cap B_1) + T(\O_2) - T(\O_2 \cap B_2)\\
&\le C_d \ (|\O_1 \setminus B_1| + |\O_2 \setminus B_2|) \\
& = C_d \ |\O\setminus \Theta|.
\end{aligned}
$$
Suppose now that $\O_1$ is connected and $|\O_2\cap \Theta| = 0 $, we obtain
$$
\begin{aligned}
T(\O) - T(\O\cap \Theta) &= T(\O_1) - T(\O_1\cap \Theta) + T(\O_2)\\
&\le T(\O_1) - T(\O_1\cap \Theta) + \left(\frac{|\O_2|}{\omega_d}\right)^\frac{d+2}{d} T(B)\\
&\le T(\O_1) - T(\O_1\cap \Theta) + \frac{T(B)}{\omega_d} |\O\setminus\Theta|.
\end{aligned}
$$
From this preliminary analysis, it only remains to consider connected sets $\O$ that intersect with both balls $B_1$ and $B_2$ of $\Theta$. 

We now tackle the case where $B_1$ and $B_2$ are far from each other. The intuition behind this argument is that since $\O$ is really stretched with a fixed volume, it has to be very thin somewhere and the torsion function has to be small. Then using ideas similar to those of Bucur and Mazzoleni in \cite{BM15} we can cut $\O$ into two different pieces in a controlled manner that we can then compare to each ball separately.\\
Suppose that $\mathrm{dist}(B_1,B_2) \ge L$ for some $L > 0$ large and we define $H_L$ the slab orthogonal to the axis passing through the centers $x_1$ and $x_2$ of the balls centered in $(x_1+x_2)/2$. We denote by $S_R(x) \subset H_L$ the slab of width $R>0$ centered in $x\in[x_1, x_2]$. Since $|\O| = \omega_d$, for every $R\le L$ there exists an $x_R \in (x_1,x_2)$ such that $|S_R(x_R)\cap \O| \le \omega_d \frac RL$. Therefore, for every $x \in S_{\frac R2}(x_R)$, 
$$
\int_{B(x,\frac R2)} w_\O \le C_d \frac RL
$$
Now we apply Lemma \ref{lm:subharmonic} in the ball $B(x,\frac{R}{2})$, so that
$$ \forall x \in S_{\frac R2}(x_R),\ w(x) \le C_d\ (\frac{1}{L\, R^{d-1}} + R^2), $$
and this bound holds for every $R$ in $(0,L)$, we therefore optimize in $R$ and we get that, if $L$ is large enough, for 
$$ R_L = \left(\frac{d-1}{2L}\right)^\frac{1}{d+1} $$
the following bound holds
$$
\forall x \in S_{\frac {R_L}2}(x_{R_L}),\ w(x) \le C_d\ {R_L}^2
$$
Now, we apply \cite[Corollary 4.3]{BM15}, using the same notations, for $c=1/2$, we obtain that for $L$ large enough, $R_L\le r_0$ and $C_d R_L^2 \le C_0 R_L$ which leads to
$$ T(\O) - T(\O\setminus S_{\frac {R_L}2}(x_{R_L})) \le (|\O| - |\O\setminus S_{\frac {R_L}2}(x_{R_L})) | = |\O\cap S_{\frac {R_L}2}(x_{R_L})|\le |\O\setminus \Theta|. $$
Finally, by definition the set $\O \setminus S_{\frac {R_L}2}(x_{R_L}) = \O_1\sqcup\O_2$ is disconnected and satisfies $|\O_1\cap B_2| = |\O_2 \cap B_1| = 0$ and our preliminary analysis gives
$$ T(\O) - T(\O\cap\Theta) \le C_d |\O\setminus \Theta|.$$
As a consequence we can now fix some $L > 0$ (depending only on the dimension) and we assume that $B_1$ and $B_2$ both lie in the ball of radius $L/2$ centered at $0$, if $\mathrm{diam}(\O) > 2L$, we can apply the exact same procedure in every direction $e_i$ of $\Rd$ between the cuboids of length $5/4L$ and $7/4L$ because the set $\O$ must exit the ball of radius $L$. Doing so, we construct a set $\tilde \O$ of diameter $\mathrm{diam}(\tilde \O) \le 2L$ satisfying
$$
T(\O) - T(\tilde \O) \le d |\O\setminus \Theta|.
$$
Finally, it only remains to consider the sets $\O$ that have diameter smaller than $\mathrm{diam}(\O)\le2L$. It is in this step that we lose the exponent $1$.\\
We begin by showing
\be \label{eq:OtoOuT}
T(\O) - T(\O\cap\Theta) \le T(\O\cup\Theta) - T(\Theta).
\ee
Consider $w_{\O\cap\Theta},\ w_{\O\cup\Theta},\ w_{\Theta}$ the respective torsion functions of $\O\cap\Theta,\ \O\cup\Theta,\ \Theta$, and we denote by $h = w_\O - w_{\O\cap\Theta}$ and $\tilde h = w_{\O\cup\Theta} - w_{\Theta}$ defined respectively on $\O\cap\Theta$ and $\Theta$. Both functions $h$ and $\tilde h$ are harmonic and non negative by construction. And, on $(\partial \O )\cap \Theta$, $h = 0 \le \tilde h$ and on $\O\cap (\partial \Theta)$, $h = w_\O \le w_{\O\cup\Theta} = \tilde h$ so that on the set $\O\cap\Theta$, $\tilde h \ge h$. We can then directly compute 
$$ 
\begin{aligned}
T(\O) - T(\O\cap\Theta)  &= \int_{\O\setminus \Theta} w_\O + \int_{\O\cap\Theta} h\\
& \le \int_{\O\setminus \Theta} w_{\O\cup\Theta} + \int_{\O\cap\Theta} \tilde h\\
& \le T(\O\cup \Theta) + T(\Theta).
\end{aligned}
$$
Now apply a cylindrical Steiner symmetrization to the set $\O\cup\Theta$ with respect to the axis passing through the center of the balls $B_1$ and $B_2$. We obtain a simply connected set with cylindrical symmetry $\O^*$ that satisfies $T(\O^*)\ge T(\O)$, $|\O^*\setminus \Theta| = |\O\setminus \Theta|$ and $\mathrm{diam}(\O^*) \le 2L$. We now make the following manipulation,  for $h^* = w_{\O^*} - w_\Theta$, which is harmonic on $\Theta$,
$$
\begin{aligned}
T(\O^*) - T(\Theta) &= \int_{\O^*\setminus \Theta} w_{\O^*} + \int_\Theta h^*\\
&\le C_d \ (|\O\setminus \Theta| + \sup_{\partial\Theta} h)\\ 
& = C_d \ (|\O\setminus \Theta| + \sup_{\partial\Theta} w_{\O^*}).
\end{aligned}
$$

The set $\O^*$ satisfies the hypotheses of Proposition \ref{prop:holderestimate}, so that for an $M$ depending only on the dimension and $L$ and $0<\alpha_d<1$,
$$
\forall x \in \O^*,\ w_{\O^*}(x) \le M \, \mathrm{dist}(x,\partial \O^*)^{\alpha_d}
$$
This gives that if $\delta = \sup_{\partial\Theta} w_{\O^*}$, necessarily, the ball $B_\delta = B(x_0,r^*)$ of radius $r^* = (M^{-1}\delta)^{1/\alpha_d}$ centered at the point of maximum $x_0$ is included in $\O^*$. Therefore, $|\O^*\setminus \Theta| \ge |B_\delta\setminus \Theta|$ and since the center point of $B_\delta$ is on the boundary of $\Theta$, the worst possible case for $|B_\delta\setminus \Theta|$ is if the two balls $B_1$ and $B_2$ are tangent and $x_0$ is the point of contact. In this case, a direct computation gives that $|B_\delta\setminus \Theta| = O(\delta^{d+1})$. 
And finally we obtain 
$$ T(\O^*) - T(\Theta) \le C_d |\O\setminus \Theta|^\frac{\alpha_d}{d+1}.$$
\end{proof}
We may now tackle the proof of Theorem \ref{TH1}.

\begin{proof} (of Theorem \ref{TH1})

Take $\Omega$ in $\A$ and $\Theta$ a disjoint union of two balls of volume $\omega_d /2$ and fix $k\ge 1$.

First, if $\lambda_2(\Omega) > 2\, \lambda_2(\Theta)$, we know from the spectral inequality given by Lemma \ref{lm:CY} that 
$$
|\lambda_k(\Omega)-\lambda_k(\Theta)|\le k^\frac{2}{d} \left(1+\frac{4}{d}\right) (\lambda_2(\Omega)+\lambda_2(\Theta)),
$$ 
and then we can compute
$$
\begin{aligned}
|\lambda_k(\Omega)-\lambda_k(\Theta)| &\le 2\, \left(1+\frac{4}{d}\right)\, k^\frac{2}{d}\,\lambda_2(\Omega)\\
&\le 2^\frac{d+2}{d+1}\,\left(1+\frac{4}{d}\right)\,k^\frac{2}{d}\, \lambda_2(\Omega)^{1-\frac{1}{d+1}}\,(\lambda_2(\Omega)-\lambda_2(\Theta))^\frac{1}{d+1}\\
&\le C\,k^{2+\frac{4}{d}}\, \lambda_2(\Omega)^{1-\frac{1}{d+1}}\,(\lambda_2(\Omega)-\lambda_2(\Theta))^\frac{1}{d+1}.
\end{aligned}
$$
Assume now $\lambda_2(\Omega) \le 2\, \lambda_2(\Theta)$, by Lemma \ref{lm:B}, since $\Omega\cap\Theta\subset \Omega$ and $\Omega\cap\Theta\subset \Theta$, we obtain 
$$
\begin{aligned}
\frac{1}{\lambda_k(\Omega)}-\frac{1}{\lambda_k(\Omega\cap\Theta)}&\le \exp(1/(4\pi))\,k\,\lambda_k(\Omega)^{d/2} (T(\Omega)-T(\Omega\cap\Theta)),\\
\frac{1}{\lambda_k(\Theta)}-\frac{1}{\lambda_k(\Omega\cap\Theta)}&\le \exp(1/(4\pi))\,k\,\lambda_k(\Theta)^{d/2} (T(\Theta)-T(\Omega\cap\Theta)).
\end{aligned}
$$ 
We combine this with Lemma \ref{lm:CY} and the minimality of $\Theta$ for $\lambda_2$ to get 
$$
\begin{aligned}
\frac{1}{\lambda_k(\Omega)}-\frac{1}{\lambda_k(\Omega\cap\Theta)}&\le \exp(1/(4\pi))\,k^2\,\left(1+\frac{4}{d}\right)^\frac{d}{2}\,\lambda_2(\Omega)^{d/2}  \big(T(\Omega)-T(\Omega\cap\Theta) \big),\\
\frac{1}{\lambda_k(\Theta)}-\frac{1}{\lambda_k(\Omega\cap\Theta)}&\le \exp(1/(4\pi))\,k^2\,\left(1+\frac{4}{d}\right)^\frac{d}{2}\,\lambda_2(\Omega)^{d/2}  \big(T(\Theta)-T(\Omega\cap\Theta) \big).
\end{aligned}
$$
And we deduce, adding these two inequalities 
$$
\left| \frac{1}{\lambda_k(\Omega)}-\frac{1}{\lambda_k(\Theta)}\right|\le \exp(1/(4\pi))\,k^2\,\left(1+\frac{4}{d}\right)^\frac{d}{2}\,\lambda_2(\Omega)^{d/2} \big(T(\Omega) + T(\Theta) -2\,T(\Omega\cap\Theta)\big),
$$
which rewrites, using Lemma \ref{lm:CY} again, as
$$
|\lambda_k(\Omega)-\lambda_k(\Theta)| \le C_d k^{2+\frac{4}{d}}\,\lambda_2(\Omega)^{2+\frac{d}{2}}\, \big(T(\Omega) + T(\Theta) -2\,T(\Omega\cap\Theta) \big).
$$

It remains to estimate the factor $T(\Omega) + T(\Theta) -2\,T(\Omega\cap\Theta)$, to do so, we rewrite it 
$$
T(\Theta) -T(\Omega)+2\, \big(T(\Omega)-\,T(\Omega\cap\Theta) \big).
$$

In Lemma \ref{lm:torvol} we have already computed 
$$
T(\Omega)-T(\Omega\cap\Theta) \le C_d\,|\Omega\setminus \Theta|^\frac{\alpha_d}{d+1}.
$$
Choosing rightfully the two balls that compose $\Theta$, so that they are close to a minimizer for the {\it Fraenkel 2-asymmetry} of $\O$ defined in Equation \eqref{2Fasym}, we get 
$$
|\Omega\setminus \Theta|\le \omega_d \,\mathcal{F}_2(\Omega).
$$
Then we have by the quantitative Krahn-Szeg\"o inequality \eqref{eq:qKSI},
$$
|\Omega\setminus \Theta|\le C_d\, \lambda_2(\Theta)^{-\frac{1}{d+1}}\,(\lambda_2(\Omega)-\lambda_2(\Theta))^\frac{1}{d+1}.
$$
Combining the two we get that 
$$
2\,(T(\Omega)-T(\Omega\cap\Theta))\le C_d \,\lambda_2(\Theta)^{-\frac{\alpha_d}{(d+1)^2}}\,(\lambda_2(\Omega)-\lambda_2(\Theta))^\frac{\alpha_d}{(d+1)^2}.
$$
Now, from the Kohler-Jobin inequality for the second eigenvalue (\ref{lm:KJ2}), and Saint Venant inequality \ref{eq:SVI} and since $\lambda_2(\Omega) \le 2\, \lambda_2(\Theta)$, we obtain
$$
\begin{aligned}
T(\Theta)-T(\Omega)&\le T(\Omega)\,\left[\left(\frac{\lambda_2(\Omega)}{\lambda_2(\Theta)}\right)^\frac{d+2}{2}-1\right]\\
&\le C_d \lambda_2(\Theta)^{-\frac{\alpha_d}{(d+1)^2}}\,[\lambda_2(\Omega)-\lambda_2(\Theta)]^\frac{\alpha_d}{(d+1)^2}.
\end{aligned}
$$
Finally, we obtain using the upper bound $\lambda_2(\Omega)\le2\,\lambda_2(\Theta)$,
$$
\begin{aligned}
\left|\lambda_k(\Omega)-\lambda_k(\Theta)\right|&\le C_d\,k^{2+\frac{4}{d}}\,\lambda_2(\Omega)^{2+\frac{d}{2}}\,\lambda_2(\Theta)^{-\frac{\alpha_d}{(d+1)^2}}\,(\lambda_2(\Omega)-\lambda_2(\Theta))^\frac{\alpha_d}{(d+1)^2}\\
&\le C_d\,k^{2+\frac{4}{d}}\,\lambda_2(\Omega)^{1-\frac{\alpha_d}{(d+1)^2}}\,(\lambda_2(\Omega)-\lambda_2(\Theta))^\frac{\alpha_d}{(d+1)^2}.
\end{aligned}
$$
\end{proof}
\section{Proof of Theorems \ref{TH2bis} and \ref{TH2}}\label{sec4}
The statement of Theorem \ref{TH2bis} is a consequence of the more technical Theorem \ref{TH2}. We will prove first Theorem \ref{TH2} and get, as a consequence, Theorem \ref{TH2bis}.

The idea is that using the two subsets $\Omega^+$ and $\Omega^-$ we are able to benefit from the sharp estimation \eqref{eq:lbd1} that was proved for the first eigenvalue and thus obtain the sharp exponent $1/2$.

\begin{proof} (of Theorem \ref{TH2})
Take $\Omega$ in $\A$ and $\Theta$ a disjoint union of two balls of volume $\omega_d/2$ and fix $k\ge 1$. Now take  $\Omega^+$ and $\Omega^-$ two disjoint open subsets of $\Omega$ such that $\lambda_2(\Omega) \ge \max(\lambda_1(\Omega^+), \lambda_1(\Omega^-))$. We know the existence of such sets by the decomposition lemma \ref{lm:decomp}. 
First, as in the proof Theorem \ref{TH1}, if $\lambda_2(\Omega) > 2\, \lambda_2(\Theta)$, we know from Lemma \ref{lm:CY} that 
$$
\begin{aligned}
|\lambda_k(\Omega^+\sqcup\Omega^-)-\lambda_k(\Theta)|&\le k^\frac{2}{d} \left(1+\frac{4}{d}\right)\, (\lambda_2(\Omega^+\sqcup\Omega^-)+\lambda_2(\Theta)),\\
&\le k^\frac{2}{d} \left(1+\frac{4}{d}\right) \,(\lambda_2(\Omega)+\lambda_2(\Theta)),
\end{aligned}
$$ 
and then, by a similar computation,
$$
|\lambda_k(\Omega^+\sqcup\Omega^-)-\lambda_k(\Theta)| \le C_d\,k^{2+\frac{4}{d}}\, \lambda_2(\Omega)^{\frac{1}{2}}\,(\lambda_2(\Omega)-\lambda_2(\Theta))^\frac{1}{2}.
$$

Now, consider $\lambda_2(\Omega) \le 2\,\lambda_2(\Theta)$.\\
We introduce for $0<\eps<1$, two bounded subsets $\Omega_\eps^+$ and $\Omega_\eps^-$ of $\Omega^+$ and $\Omega^-$ respectively, satisfying
$$
\begin{array}{l}
\lambda_1(\Omega_\eps^+) \le \lambda_1(\Omega^+) + \eps\\
\lambda_1(\Omega_\eps^-) \le \lambda_1(\Omega^-) + \eps\\
\lambda_k(\Omega_\eps^+\sqcup\Omega_\eps^-) \le \lambda_k(\Omega^+\sqcup\Omega^-) + \eps
\end{array}
$$
and consider two balls $B_\eps^+$ and $B_\eps^-$ of respective volumes $|\Omega_\eps^+|$ and $|\Omega_\eps^-|$, we will specify their position later in the proof, at the moment, they do not need to be disjoint.

Then, we can compute
$$
|\lambda_k(\Omega^+\sqcup \Omega^-) -\lambda_k(\Theta)| \le \eps+  |\lambda_k(\Omega_\eps^+\sqcup \Omega_\eps^-) -\lambda_k(B_\eps^+\cup B_\eps^-)|+ |\lambda_k(B_\eps^+\cup B_\eps^-) -\lambda_k(\Theta)|.
$$
For each term, as in the previous proof, by Lemma \ref{lm:B} we obtain the two estimates
\begin{equation}
\begin{array}{c}\label{eq:page91}
|\lambda_k(\Omega_\eps^+\sqcup\Omega_\eps^-)-\lambda_k(B_\eps^+\cup B_\eps^-)| \\
 \le C_d\, k^{2+\frac{4}{d}}\,(\lambda_2(\Omega)+\eps)^{2+\frac{d}{2}}\, (T(\Omega_\eps^+\sqcup\Omega_\eps^-) + T(B_\eps^+\cup B_\eps^-) -2\,T((\Omega_\eps^+\sqcup\Omega_\eps^-)\cap(B_\eps^+\cup B_\eps^-))),
\end{array}
\end{equation}
and,
\begin{equation}\label{eq:page92}
|\lambda_k(B_\eps^+\cup B_\eps^-)-\lambda_k(\Theta)| \le C_d\, k^{2+\frac{4}{d}}\,(\lambda_2(\Omega)+\eps)^{2+\frac{d}{2}}\,\big (T(B_\eps^+\cup B_\eps^-) + T(\Theta) -2\,T((B_\eps^+\cup B_\eps^-)\cap\Theta)\big).
\end{equation}
We start by working with Inequality \eqref{eq:page91}, since the torsion is an increasing functional for the inclusion of sets and the sets $\Omega_\eps^+$ and $\Omega_\eps^-$ are disjoint,
$$T((\Omega_\eps^+\sqcup\Omega_\eps^-)\cap(B_\eps^+\cup B_\eps^-)) \ge T(\Omega_\eps^+\cap B_\eps^+)+T(\Omega_\eps^-\cap B_\eps^-),$$
therefore, Inequality \eqref{eq:page91} rewrites as
$$
\begin{array}{ll}
|\lambda_k(\Omega_\eps^+\sqcup\Omega_\eps^-)-\lambda_k(B_\eps^+\cup B_\eps^-)|\le& C\, k^{2+\frac{4}{d}}\,(\lambda_2(\Omega)+\eps)^{2+\frac{d}{2}}\cdot\\
&\Big[T(\Omega_\eps^+) + T(B_\eps^+) - 2\,T(\Omega_\eps^+\cap B_\eps^+)\\
&+T(\Omega_\eps^-) + T(B_\eps^-) - 2\,T(\Omega^-\cap B_\eps^-)\\
&+T(B_\eps^+\cup B_\eps^-) - (T(B_\eps^+) + T(B_\eps^-))\Big],
\end{array}
$$
and we only have to estimate each term of the right hand side separately.\\
We know, from following the proof of \cite[Theorem 1.1]{BLNP23}, similarly to the proof of Theorem \ref{TH1} but in the first eigenvalue setting applied to each set $\Omega_\eps^+$ and $\Omega_\eps^-$, that we can choose the two balls $B_\eps^+$ and $B_\eps^-$ such that.
$$ 
\begin{aligned}
T(\Omega_\eps^+) + T(B_\eps^+) - 2\,T(\Omega_\eps^+\cap B_\eps^+) &\le C_d (\lambda_1(\Omega_\eps^+)-\lambda_1(B_\eps^+))^\frac{1}{2},\\
T(\Omega_\eps^-) + T(B_\eps^-) - 2\,T(\Omega_\eps^-\cap B_\eps^-) &\le C_d(\lambda_1(\Omega_\eps^-)-\lambda_1(B_\eps^-))^\frac{1}{2},
\end{aligned}
$$
Since $\Omega_\eps^+$ and $\Omega_\eps^-$ are bounded and the positions of $B_\eps^+$ and $B_\eps^-$ are only dependant on $\Omega_\eps^+$ and $\Omega_\eps^-$ respectively, one can find a vector $x\in\Rd$ such that $x+\Omega_\eps^-$ and $x+B_\eps^-$ are both disjoint of $\Omega_\eps^+$ and $B_\eps^+$. Therefore by invariance of the spectrum and the torsional rigidity by translation of connected components, we can consider that the two balls $B_\eps^+$ and $B_\eps^-$ are disjoint.\\
As a consequence, we obtain 
$$ T(B_\eps^+\cup B_\eps^-) - (T(B_\eps^+) + T(B_\eps^-)) = 0.$$
Now, since $ \max (\lambda_1(\Omega_\eps^+), \lambda_1(\Omega_\eps^-))\le \lambda_2(\Omega) +\eps$, we deduce that 
$$
\begin{aligned}
T(\Omega_\eps^+) + T(B_\eps^+) - 2\,T(\Omega_\eps^+\cap B_\eps^+) &\le C_d (\lambda_2(\Omega) + \eps-\lambda_1(B_\eps^+))^\frac{1}{2},\\
T(\Omega_\eps^-) + T(B_\eps^-) - 2\,T(\Omega_\eps^-\cap B_\eps^-) &\le C_d(\lambda_2(\Omega) + \eps-\lambda_1(B_\eps^-))^\frac{1}{2}.
\end{aligned}
$$
If $|\Omega_\eps^+|, |\Omega_\eps^-|\le \omega_d/2$ we have since the eigenvalue is decreasing for the inclusion of sets that $\lambda_1(B_\eps^+), \lambda_1(B_\eps^-)\ge \lambda_2(\Theta)$ and we can conclude. So it only remains to consider the case $|\Omega_\eps^-| < \omega_d/2 < |\Omega_\eps^+|$, we still have $\lambda_1(\O^+)\ge\lambda_2(\Theta)$, and we claim that
\begin{equation}\label{eq:claim}
\lambda_2(\Omega) + \eps-\lambda_1(B_\eps^+) \le 2\, (\lambda_2(\Omega) + \eps-\lambda_2(\Theta)),
\end{equation}
then we have showed that 
$$
|\lambda_k(\Omega_\eps^+\cup\Omega_\eps^-)-\lambda_k(B_\eps^+\cup B_\eps^-)| \le C_d k^{2+\frac{4}{d}}\,(\lambda_2(\Omega)+\eps)^{2+\frac{d}{2}} \,(\lambda_2(\Omega) + \eps -\lambda_2(\Theta))^\frac{1}{2}.
$$
To prove the claim we just compute, for $B$ any ball of radius $1$, and any $0<t<1$,
$$
\lambda_2(\Theta) = \left(\frac{1}{2}\right)^{-\frac{2}{d}}\,\lambda_1(B)\le \frac{1}{2} \, ( t^{-\frac{2}{d}} + (1-t)^{-\frac{2}{d}}) \, \lambda_1(B),$$
choosing $t = \frac{|B_\eps^+|}{|B|}$ we obtain $t^{-\frac{2}{d}} \,\lambda_1(B) = \lambda_1(B_\eps^+)$ and $(1-t)^{-\frac{2}{d}} \,\lambda_1(B) \le \lambda_1(B_\eps^-) \le \lambda_2(\Omega) +\eps$. And then, we get 
$$
\lambda_2(\Theta) \le \lambda_2(\Omega) + \eps + \frac{1}{2}(\lambda_1(B^+) - (\lambda_2(\Omega)+\eps)),
$$
which proves the claim.

We now consider Inequality \eqref{eq:page92}. Once again we need to consider two cases depending on the volume. If $|B_\eps^+| \le \omega_d/2$ and $|B_\eps^-| \le \omega_d/2$, we choose the two balls $B^1$ and $B^2$ composing $\Theta$ such that $B_\eps^+\subset B^1$ and $B_\eps^-\subset B^2$ and then we are left with estimating the term 
$$
T(B^1) - T(B_\eps^+) + T(B^2) - T(B_\eps^-).
$$
Then, using the scaling, and the fact that 
$$\lambda_1(B_1)\le \lambda_1(B_\eps^+) \le \lambda_1(\O_\eps^+)\le \lambda_2(\O) + \eps\le 2\lambda_2(\Theta) + 1,$$ 
we compute 
$$
\begin{aligned}
T(B^1) - T(B_\eps^+) &\le C_d (\lambda_1(\O_\eps^+) - \lambda_2(\Theta)) \le C_d (\lambda_2(\O) + \eps - \lambda_2(\Theta)),  \\
\end{aligned}
$$
and the same bound holds for $T(B_2) - T(B_\eps^-)$.\\
If we have $|B_\eps^-| < \omega_d/2 < |B_\eps^+|$ by choosing $B^1\subset B_\eps^+$ and $B^2\supset B_\eps^-$, we are left with estimating $T(B_\eps^+)- T(B_\eps^-)$ and with the same arguments we obtain,
$$
T(B_\eps^+)- T(B_\eps^-) \le 2\,C (\lambda_2(\Omega) +\eps -\lambda_2(\Theta)),
$$
and then we have showed 
$$
|\lambda_k(B^+\cup B^-)-\lambda_k(\Theta)| \le C_d\, k^{2+\frac{4}{d}}\,(\lambda_2(\Omega)+\eps)^{2+\frac{d}{2}}\,(\lambda_2(\Omega) +\eps -\lambda_2(\Theta)).
$$
Finally, letting $\eps$ go to $0$ we conclude the proof since  $\lambda_2(\Omega) \le 2\, \lambda_2(\Theta)$.
\end{proof}

From Theorem \ref{TH2} we can now deduce Theorem \ref{TH2bis} which comes as a direct corollary.
\begin{proof} (of Theorem \ref{TH2bis})
Consider $\Omega \in \A$ and choose two disjoints subsets $\Omega^+$, $\Omega^-$ of $\Omega$ satisfying the eigenvalue condition $\lambda_2(\Omega) \ge \max(\lambda_1(\Omega^+), \lambda_1(\Omega^-))$. Since the eigenvalues are decreasing for the inclusion of sets, it holds for all $k\ge 1$,
$$
\lambda_k(\Omega) \le \lambda_k(\Omega^+\sqcup \Omega^-).
$$
Then we can deduce 
$$
\left(\lambda_k(\Omega) - \lambda_k(\Theta)\right)_+ \le \left(\lambda_k(\Omega^+\sqcup\Omega^-) - \lambda_k(\Theta)\right)_+ \le \left|\lambda_k(\Omega^+\sqcup\Omega^-) - \lambda_k(\Theta)\right|.
$$
Finally, we can apply Theorem \ref{TH2} and obtain
$$
\left(\lambda_k(\Omega) - \lambda_k(\Theta)\right)_+ \le C_d \,k^{2+\frac{4}{d}}\,\lambda_2(\Omega)^\frac{1}{2}\,(\lambda_2(\Omega) -\lambda_2(\Theta))^\frac{1}{2}.
$$
\end{proof}
\section{Further remarks and open questions}\label{sec5}
As a conclusion, a few remarks are in order.
\begin{rem}[Sharpness of the exponent $1/2$] We have claimed that $1/2$ should be the sharp exponent in Theorem \ref{TH1}. By this we mean that one cannot find an exponent $\alpha^*> 1/2$ such that for all $\Omega \in \A$ the inequality
\begin{equation}\label{ineq:alphastar}
\left|\lambda_k(\Omega)-\lambda_k(\Theta)\right|\le C\,k^{2+\frac{4}{d}}\,\lambda_2(\Omega)^{1-\alpha^*}\,(\lambda_2(\Omega)-\lambda_2(\Theta))^{\alpha^*}
\end{equation}
holds and that the inequality should be true for $\alpha^* =1/2$. 

Even though we are not able to prove that the inequality is true for the exponent $1/2$ Theorem \ref{TH2bis} gives good hope that it should be, and we can actually show that we cannot expect a better exponent. 

Suppose that \eqref{ineq:alphastar} is true for some $\alpha^*> 1/2$ for any $\Omega$ and $k$. We can take $\Omega \in \A$, bounded. Denoting by $\Omega^{1/2}\sqcup\Omega^{1/2}$ the disjoint union of two copies of $\Omega$ rescaled to have measure $\omega_d/2$ we have for any $k$,
$$
\left| \lambda_k(\Omega) -\lambda_k(B)\right| = 2^{-\frac{1}{d}}\,\left| \lambda_{2k}(\Omega^{1/2}\sqcup\Omega^{1/2}) -\lambda_{2k}(\Theta)\right|.
$$

We then deduce from inequality \eqref{ineq:alphastar}
$$
\left|\lambda_k(\Omega)-\lambda_k(B)\right|\le 2^{-\frac{1}{d}} C\,(2k)^{2+\frac{4}{d}}\,\lambda_2(\Omega^{1/2}\sqcup\Omega^{1/2})^{1-\alpha^*}\,(\lambda_2(\Omega^{1/2}\sqcup\Omega^{1/2})-\lambda_2(\Theta))^{\alpha^*}.
$$
Which implies that 
$$
\left|\lambda_k(\Omega)-\lambda_k(B)\right|\le 2^{2+\frac{4}{d}}\,C\,k^{2+\frac{4}{d}}\,\lambda_1(\Omega)^{1-\alpha^*}\,(\lambda_1(\Omega)-\lambda_1(B))^{\alpha^*}.
$$
Since this inequality would then be true for any $\Omega \in \A$ (also the unbounded sets by continuity of the eigenvalues) and $k\ge1$, it is in contradiction with the fact that $1/2$ is the sharp exponent in Inequality \eqref{eq:lbd1}.\\
In the setting of the first eigenvalue, the sharpness of the exponent $1/2$ can actually be easily observed. Since the second eigenvalue is not critical on the ball $B$, one can find a smooth volume preserving deformation of the ball $\Omega_t$ such that $\lambda_2(\Omega_t) - \lambda_2(B) = c_2\,t + o(t)$ and by minimality of the ball for the first eigenvalue, $\lambda_1(\Omega_t) - \lambda_1(B) = c_1\,t^2 + o(t^2)$, where $c_1$ and $c_2$ are two positive constants. Then having 
$$ |\lambda_2(\Omega_t) - \lambda_2(B)| \le C (\lambda_1(\Omega_t) - \lambda_1(B))^\beta,$$
for $t$ small implies $\beta \le 1/2$.
 \end{rem}

\begin{rem}[General estimate with exponent $1/2$] We proved in Theorem \ref{TH2bis}, for any $\Omega \in \A$
\begin{equation}
\big(\lambda_k(\Omega)-\lambda_k(\Theta)\big)_+\le C\,k^{2+\frac{4}{d}}\,\lambda_2(\Omega)^{\frac{1}{2}}\,(\lambda_2(\Omega)-\lambda_2(\Theta))^\frac{1}{2}.
\end{equation}
which is of course non trivial only when $\lambda_k(\Omega)-\lambda_k(\Theta)\ge 0$. To obtain the full inequality with exponent $\frac{1}{2}$, it remains to prove an inequality of the form
\begin{equation}
\lambda_k(\Theta)-\lambda_k(\Omega)\le C_d\,k^{2+\frac{4}{d}}\,\lambda_2(\Omega)^{\frac{1}{2}}\,(\lambda_2(\Omega)-\lambda_2(\Theta))^\frac{1}{2}.
\end{equation}
Taking into account Theorem \ref{TH2} and Lemma \ref{lm:B} this reduces to the following question:\\
Is there a dimensional constant $C_d$ such that for any $\Omega \in \A$, satisfying $\lambda_2(\Omega) \le 2\,\lambda_2(\Theta)$, the following quantitative Krahn-Szeg\"o inequality holds
\begin{equation}
\frac{\lambda_2(\Omega) - \lambda_2(\Theta)}{\lambda_2(\Theta)} \ge C_d\, \left(\frac{T(\Omega) - T(\Omega_1\sqcup \Omega_2)}{T(\Omega)}\right)^2,
\end{equation}
for some $\Omega_1$ and $\Omega_2$ disjoint subsets of $\Omega$ satisfying $max\{\lambda_1(\Omega_1),\lambda_1(\Omega_2)\} \le \lambda_2(\Omega)$?
\end{rem}

\begin{rem}[About the restriction on the sets $\Omega$] In order to simplify the notations in the proofs, we fixed the volume of $\Omega$ to $\omega_d$ but all the inequalities we proved are actually scale invariant and remain true for any $\Omega$ of finite measure as long as $\Theta$ is chosen with the same measure. As for the restriction to open sets, by continuity for the $\gamma$-convergence of the eigenvalues, the theorems remain true for any quasi open set of finite measure.
\end{rem}

\section*{Acknowledgment}
The author would like to thank an anonymous referee for pointing out a mistake in the proof of Lemma \ref{lm:torvol} in a previous version of this paper.

\printbibliography
\end{document}